\documentclass{mcom-l}

\usepackage{amssymb}

\usepackage{url}

\DeclareFontEncoding{OT2}{}{} 
\newcommand{\textcyr}[1]{%
 {\fontencoding{OT2}\fontfamily{wncyr}\fontseries{m}\fontshape{n}\selectfont #1}}

\usepackage[all]{xy}

\newcommand{\Sha}{{\mbox{\textcyr{Sh}}}}
\newcommand{\Z}{{\mathbb Z}}
\newcommand{\Q}{{\mathbb Q}}
\newcommand{\R}{{\mathbb R}}
\newcommand{\F}{{\mathbb F}}
\newcommand{\CO}{{\mathcal O}}
\newcommand{\Fp}{{\mathfrak p}}
\newcommand{\To}{\longrightarrow}
\newcommand{\Hom}{\operatorname{Hom}}
\newcommand{\Sel}{\operatorname{Sel}}
\newcommand{\Reg}{\operatorname{Reg}}
\newcommand{\eps}{\varepsilon}
\newcommand{\Cl}{\operatorname{Cl}}
\newcommand{\im}{\operatorname{im}}
\newcommand{\ord}{\operatorname{ord}}
\newcommand{\BSD}{\operatorname{BSD}}
\newcommand{\tors}{{\operatorname{tors}}}
\newcommand{\an}{{\operatorname{an}}}

\newtheorem{Theorem}{Theorem}[section]
\newtheorem{Lemma}[Theorem]{Lemma}
\newtheorem{Corollary}[Theorem]{Corollary}

\theoremstyle{definition}
\newtheorem{Definition}[Theorem]{Definition}
\newtheorem{Example}[Theorem]{Example}

\theoremstyle{remark}
\newtheorem{Remark}[Theorem]{Remark}

\numberwithin{equation}{section}

\begin{document}

\title[Isogeny descent]{Explicit Isogeny Descent on Elliptic Curves}


\author{Robert L.~Miller}
\address{Quid, Inc., 733 Front Street, C1A, San Francisco, CA 94111, USA}
\email{rmiller@quid.com}

\author{Michael Stoll}
\address{Mathematisches Institut, Universit\"at Bayreuth, 95440
  Bayreuth, Germany}
\email{Michael.Stoll@uni-bayreuth.de}


\date{July 28, 2011}


\begin{abstract}
  In this note, we consider an $\ell$-isogeny descent on a pair of elliptic
  curves over~$\Q$. We assume that $\ell > 3$ is a prime. The main result
  expresses the relevant Selmer groups as kernels of simple explicit maps
  between finite-dimensional $\F_\ell$-vector spaces defined in terms of
  the splitting fields of the kernels of the two isogenies.
  We give examples of proving the $\ell$-part of the Birch and
  Swinnerton-Dyer conjectural formula for certain curves of small
  conductor.
\end{abstract}

\maketitle



\section{Introduction}

Let $E/\Q$ be an elliptic curve. Then it is known~\cite{Mordell} that
the group $E(\Q)$ of rational points on~$E$ is a finitely generated
abelian group. Its finite torsion subgroup is easily determined, but so
far there is no method known that can provably determine the rank~$r$
of~$E(\Q)$ for an arbitrary given curve~$E$.
There is another abelian group associated to~$E/\Q$, the
\emph{Shafarevich-Tate group} $\Sha(\Q, E)$. It is conjectured to be
finite for all elliptic curves; however, this is only known for curves
of analytic rank 0 or~1.

The \emph{analytic rank} is the order of vanishing of
the $L$-series $L(E,s)$ associated to~$E$ at the point $s = 1$. The
conjecture of Birch and Swinnerton-Dyer states that the analytic rank
equals the rank, and moreover gives a relation between the leading term
of the Taylor expansion of~$L(E,s)$ at~$s=1$ and various local and global
data associated to~$E$, including the order of~$\Sha(\Q, E)$.
Kolyvagin~\cite{Kolyvagin} has shown that the first part of the conjecture
holds when the analytic rank is at most~1, that in this case $\Sha(\Q, E)$
is finite, and the second part of the conjecture holds up to a rational
factor involving only primes in a certain finite set (depending on~$E$).

The two groups $E(\Q)$ and $\Sha(\Q, E)$ are related by objects that can
(in principle) be computed: for each $\ell \ge 1$, there is a finite
computable group $\Sel^{(\ell)}(\Q, E)$, the \emph{$\ell$-Selmer group}
of~$E$, and an exact sequence of abelian groups
\[ 0 \To E(\Q)/\ell E(\Q) \To \Sel^{(\ell)}(\Q, E) \To \Sha(\Q, E)[\ell] \To 0 \,. \]
If $\ell$ is a prime number, then all groups involved are $\F_{\ell}$-vector
spaces, and we obtain the relation
\[ \dim_{\F_{\ell}} \Sel^{(\ell)}(\Q, E)
     = r + \dim_{\F_{\ell}} E(\Q)[\ell] + \dim_{\F_{\ell}} \Sha(\Q, E)[\ell] \,.
\]
This can be used to obtain upper bounds for~$r$ on the one hand, but
also leads to information on~$\Sha(\Q, E)$ when $r$ is known (for example
when the analytic rank is at most~$1$). In particular, if
\[ \dim_{\F_{\ell}} \Sel^{(\ell)}(\Q, E) = r + \dim_{\F_{\ell}} E(\Q)[\ell] \,, \]
then it follows that the $\ell$-primary part of~$\Sha(\Q, E)$ is trivial.
The computation of the $\ell$-Selmer group is referred to as an
\emph{$\ell$-descent on~$E$}. How this can be done is discussed in some
detail in~\cite{SchSt}. The computation involves obtaining information
on class groups and unit groups in number fields of degree up to~$\ell^2-1$,
so this often will be infeasible when $\ell \ge 5$. Even when there is
a rational $\ell$-isogeny $\phi : E \to E'$, one usually has to deal with a field
of degree~$\ell^2 - \ell$. In this case, an alternative approach is to
compute Selmer groups associated to~$\phi$ and the dual isogeny~$\phi'$.
The information obtained still provides upper bounds for the rank~$r$
and the $\ell$-torsion of~$\Sha(\Q, E)$ (and $\Sha(\Q, E')$), but the
latter may fail to be sharp.

There is already a considerable amount of work in the literature in
specific cases. In the following, we try to give an overview, which we
do not claim to be exhaustive.
Cremona uses 2-isogeny descents in \cite{Crem} for a
large number of curves, to determine the ranks of the Mordell-Weil
groups $E(\Q)$. In addition, online notes~\cite{CremonaOnline}
describe how to extend these descents to full 2-descents, in cases where
the information gained is inconclusive. Frey~\cite{Frey84} uses
2-isogeny descent for curves of the form $y^2 = x^3 \pm p^3$ for primes
$p > 3$ to determine their ranks in terms of congruence conditions on $p$.
The general theory of 2-isogeny descents is presented in detail in
\cite[Chapter~X]{Silverman}.

Selmer~\cite{Selmer51,Selmer54} and later Cassels~\cite{Cassels59}
studied cubic twists of the cubic Fermat curve
and considered among other things
3-isogenies and the multiplication by $\sqrt{-3}$ map for these curves.
Satg\'e~\cite{PSatge} considers the 3-isogeny
from the curve given by $y^2 = x^3 + A$ to its twist $y^2 = x^3 - 27A$,
for arbitrary $A$. He determines the Selmer group of this isogeny over $\Q$
by identifying it with a certain subgroup of
$\Hom(G_{\Q\left(\sqrt{A}\right)}, \Z/3\Z)$. Jeechul Woo in his Ph.D. thesis
\cite{JWoo} works out the theory and formulae for 3-isogeny descent in the
presence of a rational 3-torsion point. Nekov\'a\v{r} considers quadratic
twists of the Fermat curve in \cite{Nekovar} and computes the Selmer groups
of rational 3-isogenies. Quer~\cite{JQuer} uses the connection between
3-isogeny Selmer groups and class groups of quadratic fields to exhibit
quadratic imaginary fields of 3-rank~6, based on
elliptic curves of the form $y^2 = x^3 + k$ with rank~12. Top~\cite{JTop}
demonstrates that the technique Quer uses applies to any elliptic curve
admitting a rational 3-isogeny. DeLong~\cite{DeLong} finds a formula for
the dimension of the Selmer groups of these 3-isogenies which also relates
the 3-ranks of the associated quadratic fields. Elkies and Rogers
\cite{ElkRog} use explicit formulas for 3-isogeny descents to construct
elliptic curves of the form $x^3 + y^3 = k$ of ranks 8 through~11
over $\Q$.

Flynn and Grattoni~\cite{FlynnGrattoni} consider
isogenies coming from rational points of prime power order and exhibit an
element of the Shafarevich-Tate group of order~13. Their PARI programs are
available at~\cite{FGprog}.

Beaver computes Selmer groups for isogenies of degree~5 in~\cite{Beaver}
and uses an explicit formula for the Cassels-Tate pairing to find nontrivial
elements of the Shafarevich-Tate group of order~5. Fisher gives general results
regarding descents over rational isogenies of degree $\ell = 5$ and $\ell = 7$
when one of the curves has a rational $\ell$-torsion point
in~\cite{TFisher, TFisherPap}, which also include tables of many specific cases.

Cremona, Fisher, O'Neil, Simon and Stoll~\cite{Cremetal}
work out the general theory for full $n$-descents. Schaefer relates in
\cite{Schaefer} Selmer groups for isogenies of abelian varieties over number
fields to class groups.
In~\cite{SchaeferDescmap} he realizes the connecting homomorphism in
Galois cohomology as evaluation of a certain function called the descent map
(see Section~\ref{gen} below) on divisors.
With Stoll he explains in~\cite{SchSt} how to do a descent
for any isogeny of odd prime power degree~$\ell^e$. One important result, which is also
used here, is that the set of bad primes can be reduced to those
above $\ell$ and those with the property that one of the corresponding
Tamagawa numbers is divisible by~$\ell$.

In this note, we will expand on~\cite{SchSt} and show how such an
$\ell$-isogeny descent can be performed with only little explicit
computation. Our main result is given in Theorem~\ref{thm}. It expresses
the relevant Selmer groups as kernels of simple explicit maps between
finite-dimensional $\F_\ell$ vector spaces. The maps and spaces are
defined in terms of the splitting fields of the kernels of the two
isogenies involved. See Section~\ref{Idea} below for an explanation of
the underlying idea. Our result makes the computation of the Selmer group
sizes very easy and straight-forward. This can be used to obtain bounds
on the rank and/or the size of the $\ell$-torsion subgroup of the
Shafarevich-Tate groups of the two curves involved. The result takes
a particularly simple form when the kernel of~$\phi$ is generated
by a rational point, see Corollary~\ref{pt}, which reproduces one of
the main results of Fisher's thesis~\cite{TFisher}.

We have used the results to finish off the verification of the full
Birch and Swinner\-ton-Dyer conjecture for a number of elliptic curves
of conductor up to~5000. For a more precise statement of this result, see
Theorem~\ref{BSD}.

We thank the anonymous referees for some useful comments.


\section{Generalities}\label{gen}

Let $\ell > 3$ be a prime, and let $E$ be an elliptic curve over~$\Q$ such
that $E$ has a rational $\ell$-isogeny. By Mazur's famous result~\cite{Mazur},
this implies that
\[ \ell \in \{5, 7, 11, 13, 17, 19, 37, 43, 67, 163\} \,. \]
We remark that everything we do in
this paper still works for $\ell = 3$, under the condition that
$E$ and~$E'$ have no special fibers of type IV or~$\text{IV}^*$. For simplicity,
we do not discuss this case in more detail. Note that a full 3-descent
as described in~\cite{Cremetal} is usually feasible (and an implementation
is available in {\sf MAGMA}, for example), so for practical purposes,
it is of particular interest to be able to deal with the case $\ell > 3$.

Let $\phi : E \to E'$ be
the isogeny, and let $\phi' : E' \to E$ be the dual isogeny. V\'elu~\cite{Velu}
gives explicit formulae for $\phi$ and $E'$ in terms of $E$. (Note that the model
given for $E'$ may not be minimal.) Our reference for the following will
be~\cite{SchSt}.

We have an exact sequence of Galois modules
\begin{equation} \label{KumSeq}
  0 \To E'[\phi'] \To E' \stackrel{\phi'}{\To} E \To 0 \,.
\end{equation}
The \emph{$\phi'$-Selmer group} $\Sel^{(\phi')}(\Q, E')$ sits in
the Galois cohomology group $H^1(\Q, E'[\phi'])$; it is defined to be
the kernel of the diagonal map in the following diagram whose (exact)
rows are obtained by taking Galois cohomology of~\eqref{KumSeq} over~$\Q$
and over all completions~$\Q_v$, respectively.
\[ \xymatrix{ E(\Q) \ar[r]^-{\delta} \ar[d]
               & H^1(\Q, E'[\phi']) \ar[r] \ar[d] \ar[rd]
               & H^1(\Q, E') \ar[d] \\
              \prod_v E(\Q_v) \ar[r]^-{\delta}
               & \prod_v H^1(\Q_v, E'[\phi']) \ar[r]
               & \prod_v H^1(\Q_v, E')
            }
\]
The \emph{Shafarevich-Tate group} $\Sha(\Q, E')$ is the kernel of the
right-most vertical map in the diagram. This leads to the exact sequence
\[ 0 \To E(\Q)/\phi'(E'(\Q)) \stackrel{\delta}{\To} \Sel^{(\phi')}(\Q, E')
     \To \Sha(\Q, E')[\phi'] \To 0 \,.
\]

By the usual yoga (see~\cite[p.~1222]{SchSt}), we find that
\[ H^1(\Q, E'[\phi']) \cong \left(K^\times/(K^\times)^\ell\right)^{(1)} \,, \]
where $K$ is the field of definition of any nontrivial point~$P$ in
the kernel~$E[\phi]$, and the superscript~$(1)$ denotes the subgroup
on which the automorphism induced by $P \mapsto aP$ acts as $z \mapsto z^a$
(for $a \in \F_\ell^\times$ such that $aP$ is in the same Galois orbit
as~$P$). If $S$ is a finite set of primes, we define\footnote{There is a
variant of this definition that
requires $K(\sqrt[\ell]{\alpha})/K$ to be unramified outside primes above
primes in~$S$. This does not make a difference when $\ell \in S$. For our
purposes, the definition given here is more convenient. Note that it will
be used later with sets~$S$ possibly not containing~$\ell$.}
\begin{align*}
  K(S, \ell) &= \{\alpha (K^\times)^\ell
                   : \text{$\ell \mid v_{\Fp}(\alpha)$ for all primes $\Fp$
                           of~$K$ not above some $p \in S$}\} \\
             &\subset \frac{K^\times}{(K^\times)^\ell} \,.
\end{align*}
Then the image of the Selmer group is contained in $K(S, \ell)^{(1)}$
if $S$ contains $\ell$ and the primes~$p$
such that $\ell$ divides one of the Tamagawa numbers $c_p(E)$ or~$c_p(E')$
(see~\cite[Prop.~3.2]{SchSt}).
Note that if $E(\Q)[\phi] \neq 0$, then we have that
$K = \Q$ and $K(S, \ell)^{(1)} = \Q(S, \ell)$.

We fix our notations by requiring that $E[\phi] \subset E(\R)$ and
$E'[\phi'] \cap E'(\R) = 0$. This puts a definite order on the pair $(E, E')$.

Let $F \in K(E)$ be the \emph{descent map}, \emph{i.e.,} $F$ has a zero of order~$\ell$
at~$P$ and a pole of order~$\ell$ at the origin~$O$ of~$E$, and is normalized
such that in terms of the local parameter $t = y/x$ at~$O$ (we fix
a globally minimal Weierstrass equation for~$E$), we have
\[ F(t) = t^{-\ell}(1 + tf(t)) \,, \]
for some power series $f(t)$ over $K$. Then the connecting homomorphism~$\delta$
in the sequence above can be identified with
\[ F : E(\Q) \To K(S, \ell)^{(1)} \,, \]
where we set $F(O) = 1$ and $F(P) = 1/F(-P)$ (if $P \in E(\Q)$).
We let $K_p = K \otimes_{\Q} \Q_p$, then there is a canonical homomorphism
$r_p : K^\times/(K^\times)^\ell \to K_p^\times/(K_p^\times)^\ell$.
The function~$F$ induces a map $F_p : E(\Q_p) \to K_p^\times/(K_p^\times)^\ell$.

We then have
\begin{equation} \label{SelEq}
   \Sel^{(\phi')}(\Q, E')
     = \{\xi \in K(S,\ell)^{(1)}
         : r_p(\xi) \in F_p(E(\Q_p)) \text{\ for all $p \in S$}\}
\end{equation}
as a subgroup of~$K(S, \ell)^{(1)}$, where $S$ is as before.

In the following, we want to make the expression on the right hand side
of Equation~\eqref{SelEq} as explicit as possible.


\section{The basic idea} \label{Idea}

We begin with a definition that is needed below.

\begin{Definition} \label{DefStar}
  For $S$ a finite set of primes and $K$ a number field, let
  \[ K(S, \ell)^*
      = \prod_{p \in S} \frac{\CO_{K,p}^\times}{(\CO_{K,p}^\times)^\ell} \,,
  \]
  where $\CO_{K,p} = \CO_K \otimes_{\Z} \Z_p$ is the $p$-adic completion
  of the ring of integers of~$K$. If $S$ and~$S'$ are finite disjoint sets of primes,
  then there is an obvious canonical map
  \[ K(S, \ell) \To K(S', \ell)^* \,. \]
  (Note that for $p \notin S$, an element of $K(S, \ell)$ always has a
  representative that is a $p$-adic unit.)
\end{Definition}

Let $S$ be the set of primes~$p$ such that $\ell$ divides one of the
Tamagawa numbers $c_p(E)$ or~$c_p(E')$, together with the prime~$\ell$.
According to equation~\eqref{SelEq} above, we need to find the subgroup
of $K(S, \ell)^{(1)}$ consisting of elements satisfying certain local
conditions at the primes $p \in S$. This will be made fairly easy if these
local conditions are of a simple nature. The simplest possible cases
certainly occur when the `local image' $F_p(E(\Q_p))$ is either trivial
or the full local group $H_p = \bigl(K_p^\times/(K_p^\times)\bigr)^{(1)}$.
Another easy situation is when the local image is exactly the part~$U_p$ of the
local group that comes from $p$-adic units, since in that case, we can
just drop~$p$ from~$S$.

\begin{Lemma} \label{LemmaGen}
  We now assume that for each $p \in S$, we are in one of the three cases
  mentioned above, and we set
  \[ S_1 = \{p \in S : F_p(E(\Q_p)) = H_p\} \quad\text{and}\quad
     S_2 = \{p \in S : F_p(E(\Q_p)) = 0\} \,.
  \]
  Then
  \[ \Sel^{(\phi')}(\Q, E')
      = \ker\bigl(\alpha : K(S_1, \ell)^{(1)} \to K(S_2, \ell)^*\bigr) \,.
  \]
  Here $\alpha$ is the canonical map from Definition~\ref{DefStar}.
\end{Lemma}

\begin{proof}
  First we show that the Selmer group is contained in $K(S_1, \ell)^{(1)}$.
  This means that for all $p \in S \setminus S_1$, the local image is
  contained in~$U_p$. But for all these primes, we have by assumption
  that the image is either trivial or equals~$U_p$, so the condition
  is satisfied.

  Now we check that the elements in the kernel of~$\alpha$ are exactly
  those that satisfy the local conditions at all $p \in S$. If
  $p \in S \setminus (S_1 \cup S_2)$, then the local image equals~$U_p$,
  and this condition is already taken care of since $p \notin S_1$.
  If $p \in S_1$, then the local image is all of~$H_p$, and therefore
  there is no condition. Finally, if $p \in S_2$, then the local
  image is trivial, which means that the Selmer group
  elements are represented by elements of~$K(S_1, \ell)$ that are
  $\ell$th powers in $K_p$. Since $p \notin S_1$, we can always find
  a representative that is a unit in~$K_p$; then the condition says
  that the image in $\CO_{K,p}^\times/(\CO_{K,p}^\times)^\ell$ is trivial.
  Since
  \[ \ker \alpha = \bigcap_{p \in S_2}
              \ker\Bigl(K(S_1, \ell)^{(1)}
                         \to \frac{\CO_{K,p}^\times}{(\CO_{K,p}^\times)^\ell}\Bigr)\,,
  \]
  the claim follows.
\end{proof}

We denote by $K'$, $F'$, $F'_p$, $H'_p$, $U'_p$ etc.\ the objects corresponding
to $K$, $F$, $F_p$, $H_p$, $U_p$ etc.\ for the dual isogeny. Then it is a fact
that there is a perfect pairing
\[ H_p \times H'_p \To \frac{\frac{1}{\ell}\Z}{\Z} \cong \F_\ell \]
(induced by cup product and the Weil pairing on $H^1$'s) such that
the images of~$F_p$ and~$F'_p$ are exact annihilators of each other.
(See~\cite[Cor.~I.2.3 and~I.3.4]{Milne}; the last statement follows from the
compatibility of the two pairings.)
In particular, $\im(F_p) = 0$ is equivalent to $\im(F'_p) = H'_p$,
and $\im(F_p) = H_p$ is equivalent to $\im(F'_p) = 0$.
In addition, we find that the $\F_\ell$-dimensions satisfy
\[ \dim H_p = \dim H'_p = \dim \im(F_p) + \dim \im(F'_p) \,. \]
If $p \neq \ell$, then by Lemma~3.8 in~\cite{Schaefer}, we have
\[ \# \frac{E(\Q_p)}{\phi'(E'(\Q_p))}
    = \# \im(F_p)
    = \frac{c_p(E)}{c_p(E')} \# E'(\Q_p)[\phi']
\]
and an analogous relation for~$\phi$. If $p = \ell$, then the last
expression has to be multiplied by $\ell^{v_\ell(\gamma')}$, where
$(\phi')^*(\omega_E) = \gamma' \omega_{E'}$ and $\omega_E$, $\omega_{E'}$
are the differentials associated to a minimal Weierstrass model.
This prompts the following definition.

\begin{Definition} \label{Defw}
  We set $w = v_\ell(\gamma')$.
\end{Definition}

We define $\gamma$ by $\phi^*(\omega_{E'}) = \gamma \omega_E$, then
$\gamma \gamma' = \ell$, and so we have $v_\ell(\gamma) = 1 - w$.
We obtain
\[ \dim \im(F_p) + \dim \im(F'_p)
    = \dim E(\Q_p)[\phi] + \dim E'(\Q_p)[\phi']
       + \begin{cases}
           0 & \text{if $p \neq \ell$,} \\ 1 & \text{if $p = \ell$.}
         \end{cases}
\]

We will need to determine~$w$. This is done in the following lemma.
We denote by $\Omega(E) = \int_{E(\R)} |\omega_E|$ the real period
of~$E$. Recall that we had fixed $E$ to be the curve with $E(\R)[\phi] \neq 0$
(and therefore, we have $E'(\R)[\phi'] = 0$).

\begin{Lemma} \label{Lemw}
  We have $\Omega(E)/\Omega(E') = \ell^w$.
\end{Lemma}

\begin{proof}
  Since $E'(\R)[\phi'] = 0$,
  $\phi'$ is an isomorphism from $E'(\R)$ to~$E(\R)$. Hence
  \[ |\gamma'| \Omega(E') = \int_{E'(\R)} |\gamma' \omega_{E'}|
                          = \int_{E'(\R)} |(\phi')^* \omega_E|
                          = \int_{E(\R)} |\omega_E| = \Omega(E) \,.
  \]
  We know that $\phi^* \omega_{E'}$ is an integral multiple of~$\omega_E$
  and that ${\phi'}^* \omega_E$ is an integral multiple of~$\omega_{E'}$;
  also $(\phi' \circ \phi)^* \omega_E = \ell \omega_E$. Therefore, $|\gamma'| = 1$
  (and $w = 0$) or $|\gamma'| = \ell$ (and $w = 1$), and the claim follows.
\end{proof}

Since we can easily compute the real periods using a system like {\sf MAGMA},
Sage or PARI-gp, $w$ can be determined for any given isogeny. In some cases,
the periods are computed with respect to the given model, so it is important
to use globally minimal models of the two curves to get correct results.

We will show below in Section~\ref{Tate}
that when $p \in S$, but $p \neq \ell$, then we always
have either trivial or full local image, and that the two cases are
distinguished by looking at the quotient $c_p(E)/c_p(E')$.
The only possible problem can therefore occur when $p = \ell$. If $\ell$
divides one of the Tamagawa numbers $c_\ell(E)$ or $c_\ell(E')$, then
the result is the same as for $p \neq \ell$. Otherwise, we see
by the above that the following holds.

\begin{samepage}
\begin{Lemma} \label{LemmaEll1}
  Assume that $\ell \nmid c_\ell(E) c_\ell(E')$.
  \begin{enumerate}
    \item If $E'(\Q_\ell)[\phi'] = 0$ and $w = 0$, then
          $\im(F_\ell) = 0$ and $\im(F'_\ell) = H'_\ell$.
    \item If $E(\Q_\ell)[\phi] = 0$ and $w = 1$, then
          $\im(F_\ell) = H_\ell$ and $\im(F'_\ell) = 0$.
  \end{enumerate}
\end{Lemma}
\end{samepage}

Since $\Q_\ell$ does not contain~$\mu_\ell$, it is not possible that
both $E(\Q_\ell)[\phi]$ and~$E'(\Q_\ell)[\phi']$ are nontrivial. The
cases that are left are therefore
\begin{itemize}
  \item $E'(\Q_\ell)[\phi'] \neq 0$ and $w = 0$ \quad and
  \item $E(\Q_\ell)[\phi] \neq 0$ and $w = 1$.
\end{itemize}
Then both local images at~$\ell$ are one-dimensional subspaces of the
two-dimensional space $H_\ell$ or~$H'_\ell$. It will turn out that
$\im(F'_\ell) = U'_\ell$ in the first and $\im(F_\ell) = U_\ell$ in
the second of the two cases, see Lemma~\ref{Ell} below.
So we will be able to compute at least
one of the two Selmer groups $\Sel^{(\phi')}(\Q, E')$ and $\Sel^{(\phi)}(\Q, E)$
easily. There is a formula due to Cassels that relates the sizes of
these two groups, see Section~\ref{SCassels} below. This allows us
to deduce the size of the other Selmer group. Cassels' formula may also
be useful when both Selmer groups can be computed by our methods,
since one of the two number fields $K$ and~$K'$ may be significantly
easier to deal with (for example because it is of lower degree). We
can then compute the easier group and deduce the size of the other one
by Cassels' formula.


\section{Tate Curves} \label{Tate}

We note that for all primes $\ell \neq p \in S$, we have that $\ell$
divides $c_p(E)$ or~$c_p(E')$, and that our assumption that $\ell > 3$ implies that
if $\ell \mid c_p(E) c_p(E')$ for any prime~$p$, then
both curves must have split multiplicative reduction at~$p$.
We can then use the Tate parametrization to obtain information on the
images of $F_p$ and~$F'_p$. This approach has also been used in some of
the earlier papers mentioned in the introduction.

Our reference for the following is~\cite[\S\,V.3]{SilvermanII}, in
particular Theorem~V.3.1.
If an elliptic curve~$E$ has split multiplicative reduction at~$p$, then
there is $q \in \Q_p^\times$ with $v(q) > 0$ such that
$E(\Q_p) \cong \Q_p^\times/q^{\Z}$. The subgroup of points with nonsingular
reduction is $E(\Q_p)^0 \cong \Z_p^\times$, and the kernel of reduction
is $E(\Q_p)^1 \cong (1 + p \Z_p)$. Therefore we find that
\[ E(\Q_p)^0/E(\Q_p)^1 \cong \Z_p^\times/(1 + p\Z_p) \cong \F_p^\times \]
(as must be the case for split multiplicative reduction) and that the
component group is
\[ \Phi_p = E(\Q_p)/E(\Q_p)^0 \cong \Z/v_p(q)\Z \]
where the isomorphism is induced by the valuation on~$\Q_p^\times$.
In particular, the Tamagawa number is $c_p(E) = v_p(q)$.
This description of~$E(L)$ carries over to all finite extensions~$L$
of~$\Q_p$.

The $\ell$-torsion subgroup of~$E$
is generated by $\mu_\ell$ and $q_\ell$, where $q_\ell^\ell = q$.
So we have a point of order~$\ell$ in~$E(\Q_p)$ if either $\mu_\ell(\Q_p)$
is nontrivial, which means $p \equiv 1 \bmod \ell$, or if $q$ is an
$\ell$th power in~$\Q_p$. In any case, the cyclic subgroups of order~$\ell$
are $\mu_\ell$ and the subgroups generated by some~$q_\ell$. The first
kind of subgroup is always defined over~$\Q_p$, the second kind only
if $q_\ell \in \Q_p^\times$.

In the first case, the corresponding isogenous curve is
$E'(\Q_p) \cong \Q_p^\times/(q^\ell)^{\Z}$, where the isogeny~$\phi$ is induced
by $z \mapsto z^\ell$. The dual isogeny~$\phi'$ is induced by the identity map,
which implies that $E(\Q_p)/\phi'(E'(\Q_p))$ is trivial. Note that in this case we
have $c_p(E') = v_p(q^\ell) = \ell v_p(q) = \ell c_p(E)$.

In the second case, the corresponding isogenous curve is
$E'(\Q_p) \cong \Q_p^\times/q_\ell^{\Z}$, with isogeny~$\phi$ induced by the
identity. The dual isogeny~$\phi'$ is induced by $z \mapsto z^\ell$, so
we have $E(\Q_p)/\phi'(E'(\Q_p)) \cong \Q_p^\times/(\Q_p^\times)^\ell$,
and $c_p(E) = \ell c_p(E')$.

This leads to the following result.

\begin{Lemma} \label{LNotEll}
  Let $p$ be a prime number.
  \begin{enumerate}
    \item If $c_p(E') = \ell c_p(E)$, then $\im(F_p) = 0$ and $\im(F'_p) = H'_p$.
    \item If $c_p(E) = \ell c_p(E')$, then $\im(F_p) = H_p$ and $\im(F'_p) = 0$.
  \end{enumerate}
  If $\ell \neq p \in S$, then we are in one of these two cases.
\end{Lemma}

\begin{proof}
  We have seen in the discussion above that if the curves $E$ and~$E'$
  have split multiplicative reduction at~$p$, then we are in one of the
  two cases given in the statement. The claims on the local images follow from
  $\im(F_p) \cong E(\Q_p)/\phi'(E'(\Q_p))$ and $\im(F'_p) \cong E'(\Q_p)/\phi(E(\Q_p))$
  and the discussion preceding the statement of the lemma.
  If $\ell \neq p \in S$, then we must have split multiplicative reduction,
  therefore the first part applies.
\end{proof}


\section{The local image at $\ell$}

As mentioned at the end of Section~\ref{Idea}, the only cases that are
left to consider are when $\ell \nmid c_\ell(E) c_\ell(E')$ and either
\[ E'(\Q_\ell)[\phi'] \neq 0 \text{\ and\ } w = 0\,, \qquad\text{or}\qquad
   E(\Q_\ell)[\phi] \neq 0 \text{\ and\ } w = 1\,.
\]
Note that the case $\ell \mid c_\ell(E) c_\ell(E')$ is taken care of
by Lemma~\ref{LNotEll}.

We now have the following result.

\begin{Lemma} \label{Ell}
  Assume that $\ell \nmid c_\ell(E) c_\ell(E')$.
  \begin{enumerate}
    \item If $E'(\Q_\ell)[\phi'] \neq 0$ and $w = 0$, then $\im(F'_\ell) = U'_\ell$
          and $U_\ell = H_\ell$.
    \item If $E(\Q_\ell)[\phi] \neq 0$ and $w = 1$, then $\im(F_\ell) = U_\ell$
          and $U'_\ell = H'_\ell$.
  \end{enumerate}
\end{Lemma}

\begin{proof}
  It suffices to prove the second assertion (say), the other one following
  by symmetry. We have $E(\Q_\ell)[\phi] \neq 0$, so the kernel
  is generated by some $P \in E(\Q_\ell)$, and
  $H_\ell \cong \Q_\ell^\times/(\Q_\ell^\times)^\ell$. Since $\ell$ does
  not divide $c_\ell(E)$, the point~$P$ must have nonsingular reduction.
  In terms of a minimal integral Weierstrass equation, the descent map
  is then given by a polynomial $f(x) + g(x) y \in \Z_\ell[x,y]$ (with
  $\deg f \le (\ell-1)/2$ and $g$ monic of degree $(\ell-3)/2$). It follows
  that $F_\ell(Q)$ is a unit for all points $Q \in E(\Q_\ell)$ that do not
  reduce to the same point as the origin or~$P$. For points in the same
  residue class as~$P$, we use that $F_\ell(Q) = 1/F_\ell(-Q)$. For points
  in the kernel of reduction, we use that $F_\ell(Q) = F_\ell(Q-P)/F_\ell(-P)$.
  So we see that $\im(F_\ell) \subset U_\ell$. Since both sides are
  of dimension~$1$ (for $\im(F_\ell)$ we use the assumption $w=1$ here),
  they must be equal. The statement on $U'_\ell$ follows by inspection of
  $H'_\ell \cong \bigl(\Q_\ell(\mu_\ell)^\times/(\Q_\ell(\mu_\ell)^\times)^\ell)^{(1)}$,
  which is generated by the images of the units $\zeta$ and $1 + \lambda^\ell$,
  where $\zeta$ is a primitive $\ell$th root of unity and $\lambda = 1-\zeta$.
\end{proof}


\section{Cassels' formula} \label{SCassels}

We make the following definition.

\begin{Definition} \label{DefS}
\[ \Sigma_1 = \{p : c_p(E) = \ell c_p(E')\} \quad\text{and}\quad
   \Sigma_2 = \{p : c_p(E') = \ell c_p(E)\} \,.
\]
Then $S = \Sigma_1 \cup \Sigma_2 \cup \{\ell\}$. Note that by the discussion
in Section~\ref{Tate}, $\Sigma_1 \cup \Sigma_2$ is exactly the set of
primes where $E$ (or equivalently, $E'$) has split multiplicative reduction.
\end{Definition}

Cassels~\cite{Cassels} has established a formula relating the sizes of
$\Sel^{(\phi')}(\Q, E')$ and $\Sel^{(\phi)}(\Q, E)$. It reads as follows.
\[ \frac{\#\Sel^{(\phi)}(\Q, E)}{\#\Sel^{(\phi')}(\Q, E')}
    = \frac{\#E(\Q)[\phi]}{\#E'(\Q)[\phi']}\,\frac{\Omega(E')}{\Omega(E)}
        \,\prod_q \frac{c_q(E')}{c_q(E)}
\]
If $E(\Q)[\phi] = 0$, then the right hand side is (by Lemma~\ref{Lemw} and
Definition~\ref{DefS}) $\ell^{\#\Sigma_2-\#\Sigma_1-w}$. Otherwise it is
$\ell^{\#\Sigma_2-\#\Sigma_1+1-w}$. (Note that $E'(\Q)[\phi'] = 0$ according to
our convention, since the nontrivial points in this kernel are not real.)
In terms of $\F_\ell$-dimensions, this says the following.

\begin{Lemma} \label{Cassels}
  \[ \dim \Sel^{(\phi)}(\Q, E) + \#\Sigma_1 + w
         = \dim \Sel^{(\phi')}(\Q, E') + \#\Sigma_2 + \dim E(\Q)[\phi]\,.
  \]
\end{Lemma}

We can combine the information from both Selmer groups in the following way.

\begin{Lemma} \label{combine}
  Let $r$ denote the rank of~$E(\Q)$ (and~$E'(\Q)$). Then we have
  \begin{align*}
    r + \dim \Sha(\Q, E)[\ell]
      &\le r + \dim \Sha(\Q, E')[\phi'] + \dim \Sha(\Q, E)[\phi] \\
      &= \dim \Sel^{(\phi')}(\Q, E') + \dim \Sel^{(\phi)}(\Q, E)
         -\dim E(\Q)[\phi] \\
      &= 2\dim \Sel^{(\phi')}(\Q, E') - \#\Sigma_1 + \#\Sigma_2 - w \\
      &= 2 \bigl(\dim \Sel^{(\phi)}(\Q, E) - \dim E(\Q)[\phi]\bigr)
           + \#\Sigma_1 - \#\Sigma_2 + w \,.
  \end{align*}
  The same bound holds for $\dim \Sha(\Q, E')[\ell]$.
  In particular, we get an upper bound for $\dim \Sha(\Q, E)[\ell]$
  and $\dim \Sha(\Q, E')[\ell]$ if
  we know the rank~$r$ and the size of one of the two Selmer groups.
\end{Lemma}

\begin{proof}
  Note first that we have $E(\Q)[\phi] = E(\Q)[\ell]$. The inclusion `$\subset$'
  is trivial. For the reverse inclusion, let $P$ be a rational $\ell$-torsion point
  on~$E$. Then $\phi(P)$ is in $E'[\phi'] \cap E'(\R) = 0$, so $P \in E(\Q)[\phi]$.

  We have the exact sequences
  \begin{gather*}
    0 \To \frac{E(\Q)}{\phi'(E'(\Q))}
      \To \Sel^{(\phi')}(\Q, E')
      \To \Sha(\Q, E')[\phi']
      \To 0 \\
    0 \To \frac{E'(\Q)}{\phi(E(\Q))}
      \To \Sel^{(\phi)}(\Q, E)
      \To \Sha(\Q, E)[\phi]
      \To 0 \\
    0 = E'(\Q)[\phi'] \To \frac{E'(\Q)}{\phi(E(\Q))}
                      \To \frac{E(\Q)}{\ell E(\Q)}
                      \To \frac{E(\Q)}{\phi'(E'(\Q))}
                      \To 0 \\
    0 \To \Sha(\Q, E)[\phi]
      \To \Sha(\Q, E)[\ell]
      \To \Sha(\Q, E')[\phi']\, ,
  \end{gather*}
  and we know that
  \[ \dim E(\Q)/\ell E(\Q) = r + \dim E(\Q)[\ell] = r + \dim E(\Q)[\phi] \,. \]
  From this, we can deduce that
  \begin{align*}
    r + \dim \Sha(\Q, E)[\ell]
      &\le  r + \dim \Sha(\Q, E')[\phi'] + \dim \Sha(\Q, E)[\phi] \\
      &= \dim \frac{E(\Q)}{\phi'(E'(\Q))}
             + \dim \frac{E'(\Q)}{\phi(E(\Q))}
             - \dim E(\Q)[\phi] \\
      &\qquad{} + \dim \Sha(\Q, E')[\phi']
            + \dim \Sha(\Q, E)[\phi] \\
      &= \dim \Sel^{(\phi')}(\Q, E')
          + \dim \Sel^{(\phi)}(\Q, E)
          - \dim E(\Q)[\phi] \\
      &= 2\dim \Sel^{(\phi')}(\Q, E') - \#\Sigma_1 + \#\Sigma_2 - w \\
      &= 2\dim \Sel^{(\phi)}(\Q, E) + \#\Sigma_1 - \#\Sigma_2 + w
          -2\dim E(\Q)[\phi] \,.
  \end{align*}
  For the last two equalities, we use Lemma~\ref{Cassels}.

  To get the bound for $\dim \Sha(\Q, E')[\ell]$, we use the exact sequence
  \[ 0 \To \Sha(\Q, E')[\phi'] \To \Sha(\Q, E')[\ell] \To \Sha(\Q, E)[\phi]\,.
     \qedhere
  \]
\end{proof}

\begin{Remark} \strut
  \begin{enumerate}
    \item \label{RkPart1}
      If $\Sha(\Q, E)$ (or equivalently, $\Sha(\Q, E')$) is finite, then the sum of the
      dimensions of $\Sha(\Q, E')[\phi']$ and $\Sha(\Q, E)[\phi]$ is even, and it
      follows that the rank~$r$ has the same parity as $\#\Sigma_1 + \#\Sigma_2 + w$.
    \item
      Recall that $\Sigma_1 \cup \Sigma_2$ is the set of primes of split
      multiplicative reduction. By \cite[Theorem 5]{DokDokParity} the root number
      $\eps(E/\Q)$ for $E$ is given in terms of local Artin symbols and the local root
      number at $\ell$:
      \[ \eps(E/\Q) = (-1)^{1+\#\Sigma_1 + \#\Sigma_2} \eps(E/\Q_\ell)
                      \prod_{p \neq \ell \text{ additive}} (-1,\Q_p(P)/\Q_p)\, \,.
      \]
      Since the parity conjecture is known
      for Selmer groups~\cite[Theorem 1.4]{DokDokBSD}
      and we are assuming $\Sha(\Q, E)$ is finite,
      the observation made in~\eqref{RkPart1} above implies that
      $(-1)^{\#\Sigma_1 + \#\Sigma_2 + w} = \eps(E/\Q)$. Combining this with the
      product formula for the Artin symbol and the fact that the Artin symbol
      is trivial for primes $p \neq \ell$ of semistable reduction
      (see~\cite{DokDokParity} again) and at~infinity according to our normalization,
      this gives us a formula for the local root number at~$\ell$:
      \[ \eps(E/\Q_\ell) = (-1)^{1-w} (-1,\Q_\ell(P)/\Q_\ell)\,. \]
      Note that since this expression only involves local data, it will
      be valid whenever $E$ has an $\ell$-isogeny defined over~$\Q_\ell$.
  \end{enumerate}
\end{Remark}


\section{The main result}

We apply Lemma~\ref{LemmaGen} to
obtain the following expressions for the Selmer groups.

\begin{Theorem} \label{thm}
  Let $\phi : E \to E'$ be an isogeny of prime degree~$\ell > 3$ of elliptic
  curves over~$\Q$, with dual isogeny $\phi' : E' \to E$, and assume that
  $E[\phi] \subset E(\R)$. Let $K$ and~$K'$ be the splitting fields
  of~$E[\phi]$ and~$E'[\phi']$, respectively, and define $\Sigma_1$, $\Sigma_2$
  and~$w$ as in Definitions \ref{DefS} and~\ref{Defw} above.

  If $\ell \nmid c_\ell(E) c_\ell(E')$, $w = 1$ and $E(\Q_\ell)[\phi] = 0$,
  then let $S_1 = \Sigma_1 \cup \{\ell\}$, else let $S_1 = \Sigma_1$.

  If $\ell \nmid c_\ell(E) c_\ell(E')$, $w = 0$ and $E'(\Q_\ell)[\phi'] = 0$,
  then let $S_2 = \Sigma_2 \cup \{\ell\}$, else let $S_2 = \Sigma_2$.

  Let
  \[ \alpha : K(S_1, \ell)^{(1)} \To K(S_2, \ell)^*
      \quad\text{and}\quad
      \beta : K'(S_2, \ell)^{(1)} \To K'(S_1, \ell)^*
  \]
  be the canonical maps.

  Then $\Sel^{(\phi')}(\Q, E') = \ker \alpha$ unless
  $\ell \nmid c_\ell(E) c_\ell(E')$, $w = 0$ and $E'(\Q_\ell)[\phi'] \neq 0$,
  and $\Sel^{(\phi)}(\Q, E) = \ker \beta$ unless
  $\ell \nmid c_\ell(E) c_\ell(E')$, $w = 1$ and $E(\Q_\ell)[\phi] \neq 0$.

  In the two excluded cases, we still have inclusions
  $\Sel^{(\phi')}(\Q, E') \subset \ker \alpha$ and
  $\Sel^{(\phi)}(\Q, E) \subset \ker \beta$, respectively.
\end{Theorem}

We see that in each case, we obtain an explicit description for at least one
of the two Selmer groups, which we can therefore determine fairly easily.
We repeat the observation that this is sufficient to obtain a bound on the
$\ell$-torsion in~$\Sha$, compare Lemma~\ref{combine}.

\begin{proof}
  We observe that in all relevant cases, the sets $S_1$ and~$S_2$
  correspond to those defined in Lemma~\ref{LemmaGen}. For primes $p \neq \ell$,
  this follows from Lemma~\ref{LNotEll}, which also covers the case $p = \ell$
  when $\ell \mid c_\ell(E) c_\ell(E')$. The remaining cases for $p = \ell$
  are dealt with in Lemmas \ref{LemmaEll1} and~\ref{Ell}. In the cases
  where we do not claim equality, we fail to take into account the local
  condition at~$\ell$ (which is of codimension~1 in~$U_\ell = H_\ell$ or
  $U'_\ell = H'_\ell$).
\end{proof}

This provides some easy bounds on the
Selmer groups. To make this more precise, we observe the following.
We denote the class number of a number field~$K$ by~$h_K$.

\begin{Lemma} \label{KSDim}
  Let $K/\Q$ be a Galois extension with Galois group a subgroup of~$\F_\ell^\times$.
  Let $S$ be a finite set of primes, and denote by $S' \subset S$ the subset
  of primes that are totally split in~$K$. If $\ell \nmid h_K$, then
  \[ \dim_{\F_\ell} K(S, \ell)^{(1)} = \#S' + 1 \]
  if $K$ is totally real but $K \neq \Q$, or $K = \Q(\mu_\ell)$ with the
  standard action on~$\mu_\ell$, and
  \[ \dim_{\F_\ell} K(S, \ell)^{(1)} = \#S' \]
  otherwise.
\end{Lemma}

\begin{proof}
  The case $K = \Q$ is clear. In general there is an exact sequence
  \[ 0 \To U_S/U_S^\ell \To K(S, \ell) \To \Cl_S(K)[\ell] \To 0 \]
  where $U_S$ is the group of $S$-units of~$K$ and $\Cl_S(K)$ is the
  $S$-class group. The assumption $\ell \nmid h_K$ implies that
  $\Cl_S(K)[\ell] = 0$. By the Dirichlet unit theorem, the group $U_S$
  has torsion-free rank $\#S_K + \#\Sigma_K - 1$, where $S_K$ is the
  set of places of~$K$ above primes in~$S$ and $\Sigma_K$ is the
  set of infinite places of~$K$.
  Let $G \subset \F_\ell^\times$ be the (cyclic) Galois group
  of~$K/\Q$. The representation of~$G$ on the $\Q$-vector space
  $U_S \otimes_{\Z} \Q$ must involve all characters of~$G$ of fixed
  order~$n \mid \#G$ with the same multiplicity~$m_n$. Let $K_n$ be
  the subfield of~$K$ of degree~$n$. Then
  \[ \sum_{k \mid n} m_k \varphi(k) = \#S_{K_n} + \#\Sigma_{K_n} - 1 \]
  for all $n \mid \#G$. We then have $\dim K(S, \ell)^{(1)} = m_{\#G}$
  (plus~1 if $\mu_\ell \subset K$ with the standard action).
  We solve this system of equations. Note that the right hand side can
  be written as a sum
  \begin{equation} \label{rhs}
    -1 + \#\Sigma_{K_n} + \sum_{p \in S} \#\{p\}_{K_n} \,.
  \end{equation}
  Since the left hand side is linear in the~$m_k$, the solution is obtained
  as a sum of solutions corresponding to the individual summands in~\eqref{rhs}.
  Let $d_p$ be the index of the decomposition group of~$p$ inside~$G$.
  Then we have $\#\{p\}_{K_n} = \gcd(n, d_p)$, and the solution for this
  right hand side is given by $m_k = 1$ if $k \mid d_p$ and $m_k = 0$ otherwise.
  An analogous statement holds for the contribution of~$\Sigma_{K_n}$,
  with $d_\infty$ the index of the subgroup generated by complex conjugation.
  The first summand~$-1$ in~\eqref{rhs} contributes $m_1 = -1$ and $m_k = 0$
  for $k > 1$. In total, we see that $m_{\#G}$ counts the number of places
  in~$S \cup \{\infty\}$ such that the decomposition group is trivial,
  which means that they split completely in~$K$. This gives the result when
  $K \neq \Q(\mu_\ell)$, since then the $S$-unit group has no $\ell$-torsion.
  For $K = \Q(\mu_\ell)$, we get an additional dimension from~$\mu_\ell$
  when $\F_\ell^\times$ acts on it in the standard way.
\end{proof}

This applies to our situation in the following way.

\begin{samepage}
\begin{Corollary} \label{simple}
  In the situation of Theorem~\ref{thm}, the following assertions hold.
  \begin{enumerate}
    \item If $\ell \nmid h_K$, then we have
          \[ \dim K(S_1, \ell)^{(1)} = \#\Sigma_1 + 1 - \dim E(\Q)[\phi] \,, \]
          and therefore (writing $\Sha[\ell]$ for either $\Sha(\Q, E)[\ell]$
          or $\Sha(\Q, E')[\ell]$)
          \[ r + \dim \Sha[\ell] \le \#\Sigma_1 + \#\Sigma_2
                                      + 2(1 - \dim E(\Q)[\phi]) - w \,.
          \]
    \item If $\ell \nmid h_{K'}$, then we have
          \[ \dim K'(S_2, \ell)^{(1)} = \#\Sigma_2 + \dim E(\Q)[\phi] \,, \]
          and therefore (writing $\Sha[\ell]$ for either $\Sha(\Q, E)[\ell]$
          or $\Sha(\Q, E')[\ell]$)
          \[ r + \dim \Sha[\ell] \le \#\Sigma_1 + \#\Sigma_2 + w \,. \]
  \end{enumerate}
\end{Corollary}
\end{samepage}

\begin{proof} \strut
  \begin{enumerate}
    \item First note that for all $p \in \Sigma_1$, $p$ is totally split
          in~$K$ (since by the discussion in Section~\ref{Tate}, we have
          $E(\Q_p)[\phi] \neq 0$). If $S_1 \neq \Sigma_1$, then
          $E(\Q_\ell)[\phi] = 0$, and so the additional element~$\ell$
          of~$S_1$ does not split completely. In the notation of Lemma~\ref{KSDim},
          we thus have $S'_1 = \Sigma_1$. Also, $K$ is totally real.
          The claim on the dimension of~$K(S_1, \ell)^{(1)}$ follows.
          By Theorem~\ref{thm}, we have
          \[ \Sel^{(\phi')}(\Q, E') \subset \ker \alpha \subset K(S_1, \ell)^{(1)} \,,\]
          so $\dim \Sel^{(\phi')}(\Q, E') \le \#\Sigma_1 + 1 - \dim E(\Q)[\phi]$.
          Lemma~\ref{combine} now gives the estimate on $r + \dim \Sha[\ell]$.
    \item In the same way as in part~(1), we find that $S'_2 = \Sigma_2$.
          Now $K'$ is not totally real, so Lemma~\ref{KSDim} gives the
          dimension of~$K'(S_2, \ell)^{(1)}$ as stated. The estimate
          then follows using Lemma~\ref{combine} as in part~(1). \qedhere
  \end{enumerate}
\end{proof}

\begin{Example} \label{kronecker-weber}
  Let $\ell \ge 11$ be a prime, and let $\phi : E \to E'$ be an $\ell$-isogeny
  of elliptic curves of conductor~$\ell^2$ (such that $E[\phi] \subset E(\R)$
  as usual). By work of Mazur~\cite{Mazur} we have $\ell \le 163$, and we find
  that in fact this applies to exactly the following curves~$E$ (with
  $\ell = 11, 19, 43, 67, 163$):
  \begin{gather*}
    \text{121a2, 121b1, 121c2, \quad 361a1, \quad 1849a1, \quad 4489a1
           \quad and \quad 26569a1} \,.
  \end{gather*}
  (We use the labeling of the Cremona database~\cite{CremonaDB}.)

  By \cite[Proposition VII.4.1]{Silverman}, the points in $E[\phi]$
  and~$E'[\phi']$ are defined over an abelian extension of~$\Q$ of degree
  dividing~$\ell-1$ and only ramified at~$\ell$. By the Kronecker-Weber theorem
  \cite[Theorem V.1.10]{Neukirch}, all such number fields are contained
  in~$\Q(\mu_\ell)$. The field $K = \Q(P)$ for any
  $P \in E[\phi] \setminus \{0\}$ is totally real, hence contained in the maximal
  totally real subfield~$\Q(\mu_\ell)^+$ of~$\Q(\mu_\ell)$. By \cite{VandVer},
  for example, it is known that $\ell$ does not divide the class number
  of~$\Q(\mu_\ell)^+$. Since $[\Q(\mu_\ell)^+ : K]$ divides $\ell-1$ and hence
  is coprime to~$\ell$, this implies that $\ell \nmid h_K$. If $\ell$ is a
  regular prime, \emph{i.e.,}~$\ell \neq 67$, then $\ell \nmid h_{K'}$ as well.

  Since $K$ is totally ramified at~$\ell$, we have that
  $[\Q_\ell(P) : \Q_\ell] = [K : \Q]$. Again by work of Mazur, $E$ has
  no rational $\ell$-torsion, implying that $K \neq \Q$. This implies
  $P \notin E(\Q_\ell)$, so $E(\Q_\ell)[\phi] = 0$. In the same way,
  we see that $E'(\Q_\ell)[\phi'] = 0$.

  Since $E$ has additive reduction at~$\ell$ and good reduction everywhere
  else, we find $\Sigma_1 = \Sigma_2 = \emptyset$. If $\ell = 67$, we verify
  that $w = 1$. Corollary~\ref{simple} now shows that
  \[ r + \dim_{\F_\ell} \Sha(\Q, E)[\ell] \le w \,. \]
  We then verify that the rank is equal to~$w$, and we see that
  $$\Sha(\Q, E)[\ell] = \Sha(\Q, E')[\ell] = 0\,.$$
\end{Example}

\begin{Example}
  Consider $E = \text{294a1}$ with $7$-isogenous curve $E' = \text{294a2}$.
  We find $\Sigma_1 = \emptyset$, $\Sigma_2 = \{2\}$, $w = 1$, and the rank $r = 0$.
  We have $K = \Q(\mu_7)^+$ (the maximal real subfield of~$\Q(\mu_7)$)
  and $K' = \Q(\sqrt{-7})$. Note that Corollary~\ref{simple} gives a bound
  of~$2$ for the dimension
  of $\Sha(\Q, E)[7]$. So we need to look more carefully to prove that there
  is no 7-torsion in~$\Sha(\Q, E)$. According to Theorem~\ref{thm},
  \[ \Sel^{(\phi)}(\Q, E)
      = \ker\bigl(\beta : K'(\{2\},7)^{(1)} \to K'(\{7\},7)^*\bigr) \,.
  \]
  The group $K'(\{2\}, 7)$ is generated by (the classes of) $1 + \sqrt{-7}$
  and $1 - \sqrt{-7}$; the two are swapped by the nontrivial automorphism.
  Therefore $K'(\{2\}, 7)^{(1)}$ is generated by their quotient. We check
  that $\frac{1 + \sqrt{-7}}{1 - \sqrt{-7}}$ is not a seventh power
  in~$\Q_7(\sqrt{-7})$. This implies that $\beta$ is injective, hence
  $\Sel^{(\phi)}(\Q, E) = 0$. Lemma~\ref{combine} then gives the bound
  \[ \dim \Sha(\Q, E)[7] \le 0 + 0 - 1 + 1 = 0 \,. \]
\end{Example}


\section{When there is rational $\ell$-torsion}

For the sake of completeness, we now consider the case that $E(\Q)[\phi] \neq 0$.
From Theorem~\ref{thm} and Lemma~\ref{Cassels}, we obtain the following,
which essentially reproduces Theorem~1 of Tom Fisher's thesis~\cite{TFisher}
(note that $\ell = 5$ or~$7$ if $E(\Q)[\phi] \neq 0$). Our version has
the slight advantage that it expresses the relevant data directly in terms of
the curves in question, whereas Fisher uses the parameter $\lambda \in X_1(\ell)$
corresponding to a generator of~$E[\phi]$.

\begin{Corollary} \label{pt}
  Assume that we have a nontrivial point $P$ in $E(\Q)[\phi]$.
  Let $S_1 = \Sigma_1$, and set $S_2 = \Sigma_2$ if $w = 1$ or $\ell \in \Sigma_1$,
  $S_2 = \Sigma_2 \cup \{\ell\}$ otherwise. Let
  \[ \alpha : \Q(S_1, \ell) \To \Q(S_2, \ell)^* \]
  be the canonical map.
  Then $\Sel^{(\phi')}(\Q, E') = \ker \alpha$, and
  \[ r + \dim_{\F_\ell} \Sha(\Q, E)[\ell]
     \le 2 \dim_{\F_\ell} \ker \alpha - \#\Sigma_1 + \#\Sigma_2 - w
     \le \#\Sigma_1 + \#\Sigma_2 - w \,.
  \]
\end{Corollary}

\begin{proof}
  The result on the Selmer group follows from Theorem~\ref{thm}.
  (Note that $E(\Q)[\phi] \neq 0$ implies $E(\Q_\ell)[\phi] \neq 0$
  and therefore $E'(\Q_\ell)[\phi'] = 0$. Recall that for $p \in \Sigma_2$,
  we have $E[\phi] \cong \mu_\ell$ as $\Q_p$-Galois modules. Since
  in the present situation $E[\phi]$ consists of rational points, this
  is only possible if $p \equiv 1 \bmod \ell$. In particular, $\ell \notin \Sigma_2$.)
  The general bound of Lemma~\ref{combine} and the trivial fact
  that $\dim \Q(S, \ell) = \#S$ then give the estimates.
\end{proof}

Note that for $p \in \Sigma_2$ (which implies that $p \equiv 1 \bmod \ell$),
\[ \Z_p^\times / (\Z_p^\times)^\ell \cong \mu_\ell(\F_p) \cong \Z/\ell\Z \,, \]
where the first isomorphism is given by $x \mapsto x^{(p-1)/\ell} \bmod p$.
The second isomorphism depends on the choice of a generator of $\mu_\ell(\F_p)$.
Note also that
\[ \Z_\ell^\times / (\Z_\ell^\times)^\ell \cong \Z/\ell\Z \]
via $x \mapsto (x^{\ell-1}-1)/\ell \bmod \ell$.

The following examples can also be found in Tables 3 and~4 in the appendix
of~\cite{TFisher}. They are given here to illustrate the simplicity of the
computations.

\begin{Example}
  Consider curve $E = \text{50b1}$, with $5$-isogenous curve 50b3.
  Note that $E(\Q)[5] \neq 0$. We have
  $\Sigma_1 = \{2\}$, $\Sigma_2 = \emptyset$, and $w = 1$. By
  Corollary~\ref{pt}, we have $\dim \Sel^{(\phi')}(\Q, E') = 1$.
  The rank is zero, so we find the bound
  \[ \dim \Sha(\Q, \text{50b1})[5] \le 2 - 1 + 0 - 1 = 0 \,. \]
\end{Example}

\begin{Example}
  For curve $E = \text{174b1}$, which is $7$-isogenous to 174b2, we have
  $E(\Q)[7] \neq 0$ and find
  $\Sigma_1 = \{2, 3\}$, $\Sigma_2 = \{29\}$; also $w = 1$. The group $\mu_7(\F_{29})$
  is generated by~$16$; we have
  \[ 16^0 = 1,\; 16^1 = 16,\; 16^2 = 24,\; 16^3 = 7,\; 16^4 = 25,\;
     16^5 = 23,\; 16^6 = 20 \,.
  \]
  If we identify $\Z_{29}^\times/(\Z_{29}^\times)^7 \cong \mu_7(\F_{29})$
  with $\Z/7\Z$ by sending $16$ to~$1$, then
  $2$ is mapped to~$1$, and $3$ is mapped to~$5$. So $\alpha$ is
  surjective and $\dim \ker \alpha = 1$.
  We obtain the bound (the rank is again zero)
  \[ \dim \Sha(\Q, \text{174b1})[7] \le 2 - 2 + 1 - 1 = 0 \,. \]
\end{Example}

\begin{Example}
  We now consider $E = \text{294b2}$ with $E(\Q)[7] \neq 0$ and
  $7$-isogenous curve~294b1. We find $\Sigma_1 = \{2,3\}$, $\Sigma_2 = \emptyset$,
  and $w = 0$. By Corollary~\ref{pt}, we have
  \[ \Sel^{(\phi')}(\Q, E')
        = \ker\bigl(\Q(\{2,3\}, 7) \to \Z_7^\times/(\Z_7^\times)^7
                                   \cong \Z/7\Z\bigr) \,.
  \]
  The map is given by $a \mapsto (a^6-1)/7$ (mod 7). Since both $2$ and~$3$
  have nontrivial image, we have $\dim \Sel^{(\phi')}(\Q, E') = 1$, and
  (the rank is zero)
  \[ \dim \Sha(\Q, \text{294b2})[7] \le 2 - 2 + 0 - 0 = 0 \,. \]
\end{Example}


\section{Application to the Birch and Swinnerton-Dyer Conjecture}

This work was motivated by a long-running project led by William Stein
(and in whose inception the second author of this paper was involved)
aiming at verifying the Birch and Swinnerton-Dyer conjecture completely
for all elliptic curves of analytic rank 0 or~1 and of moderate conductor.
The conjecture states that the analytic rank of an elliptic curve $E/\Q$
(which is the order of vanishing of the $L$-series $L(E,s)$ at $s=1$)
equals the Mordell-Weil rank~$r$ of~$E$, and that the following formula
for the leading term in the Taylor expansion at~$s=1$ holds:
\[ \frac{1}{r!} L^{(r)}(E,1)
     = \frac{\Omega(E) \Reg(E) \prod_p c_p(E) \, \#\Sha(\Q,E)}{(\#E(\Q)_\tors)^2} .
\]
Here $\Omega(E)$ is the real period of~$E$, $\Reg(E)$ is the regulator of
the height pairing on~$E(\Q)$, $c_p(E)$ are the Tamagawa numbers, and
$E(\Q)_\tors$ is the torsion subgroup of~$E(\Q)$. We call
\[ \#\Sha(\Q,E)_\an
    = \frac{1}{r!} L^{(r)}(E,1)
        \cdot \frac{(\#E(\Q)_\tors)^2}{\Omega(E) \Reg(E) \prod_p c_p(E)}
\]
the \emph{analytic order of $\Sha(\Q,E)$}. Then, assuming the rank conjecture
\[ \ord_{s=1} L(E,s) = r , \]
the second part of the conjecture asserts that $\#\Sha(\Q, E) = \#\Sha(\Q, E)_\an$.

Recall that Kolyvagin's work~\cite{Kolyvagin} proves that if~$E$ is a
modular elliptic curve over~$\Q$ and if the analytic rank of~$E$ is~0 or~1,
then the analytic rank equals the algebraic rank, $\Sha(\Q, E)$ is finite,
and its order is bounded above at all but a finite number of primes by
an explicitly computable quantity. By \cite{modularity1} and \cite{modularity2},
the modularity hypothesis may be removed from these results. It is also
known that $\#\Sha(\Q, E)_\an$ is a rational number in this case.
For such curves~$E$ and arbitrary primes~$\ell$, it then makes sense to define
\[ \BSD(E, \ell) \iff \ord_\ell \#\Sha(\Q, E) = \ord_\ell \#\Sha(\Q, E)_\an . \]
The full BSD conjecture for~$E$ is then equivalent to $\BSD(E, \ell)$ for
all primes~$\ell$.

The work in~\cite{RMiller}, which makes use of various previous results
by a number of people (see~\cite{RMiller} for references), combines
Kolyvagin's result with explicit calculations involving Heegner points,
Iwasawa theory and explicit descents to show that for such curves of
conductor less than~5000, if the mod-$\ell$ Galois representation coming from~$E$
is irreducible, then $\BSD(E, \ell)$ holds.

In the reducible case there is an $\ell$-isogeny $E \to E'$, and we have
the following result.

\begin{Theorem} \label{BSD}
  Let $\ell$ be a prime (not necessarily $\ge 5$),
  let $\phi : E \to E'$ be an isogeny of degree~$\ell$ of elliptic curves
  over~$\Q$, and assume that the analytic rank of $E$ is 0 or~1 and the
  conductor of~$E$ is less than~5000. If the $\ell$-primary parts
  of~$\Sha(\Q, E)$ and~$\Sha(\Q, E')$ are predicted by the Birch and
  Swinnerton-Dyer conjecture to be trivial, then they are indeed trivial,
  and so $\BSD(E, \ell)$ holds.
\end{Theorem}
\begin{proof}
  General results collected in~\cite{RMiller}, together with results of
  \cite{TFisher,TFisherPap} show that the claim holds for all but nineteen
  isogeny classes containing a curve~$E$ such that $E$ has analytic rank 0 or~1
  and such that there is an $\ell$-isogeny $E \to E'$.
  For eight of the remaining classes, the predicted order of both
  Shafarevich-Tate groups is trivial at~$\ell$. They are represented by the
  following pairs $(E, \ell)$ with $E[\phi] \subset E(\R)$:
  \begin{gather*}
     (\text{121b1}, 11), \quad (\text{361a1}, 19), \quad
     (\text{441d1}, 7), \quad (\text{784h1}, 7), \quad (\text{1849a1}, 43), \\
     (\text{3025a1}, 11), \quad (\text{3136r1}, 7), \quad\text{and}\quad
     (\text{4489a1}, 67).
  \end{gather*}
  For the other eleven cases, more
  is needed than the techniques discussed here, see the discussion below.

  Four of the eight cases above can be found in Example~\ref{kronecker-weber},
  which shows that $\Sha(\Q, E)[\ell]$ and $\Sha(\Q, E')[\ell]$ are trivial
  in each case. This leaves
  the four pairs $(\text{441d1}, 7),$ $(\text{784h1}, 7),$
  $(\text{3025a1}, 11)$ and $(\text{3136r1}, 7).$ For all four curves
  we have that $\Sigma_1 = \Sigma_2 = \emptyset$ and $w = 1$. One can verify that
  $\ell \nmid h_K$ and $r = 1$ in each case and so by Corollary~\ref{simple},
  we find that $\Sha(\Q, E)[\ell] = \Sha(\Q, E')[\ell] = 0$.
\end{proof}

The remaining cases where the analytic rank of $E$ is at most~1 and the
conductor is at most~5000 are
\[ E \in \{\text{570l1},\, \text{870i1},\, \text{1050o1},\,\text{1938j1},\,
           \text{1950y1},\, \text{2370m1},\, \text{2550be1},\, \text{3270h1}\}
\]
for $\ell = 5$ and $E \in \{\text{546f1},\, \text{858k1},\, \text{1230k1}\}$ for
$\ell = 7$. In these cases the Birch and Swinnerton-Dyer conjecture implies
that
\[ \Sha(\Q, E)(\ell) = 0\quad \text{but}\quad \Sha(\Q, E')(\ell) = (\Z/\ell\Z)^2\,.\]
In this situation an $\ell$-isogeny descent shows that
$\Sha(\Q, E')[\ell] = (\Z/\ell\Z)^2$,
as found in~\cite{TFisher} or using the above methods. However this shows neither that
\[ \Sha(\Q, E)[\ell] = 0 \quad \text{ nor that }\quad
   \Sha(\Q, E')(\ell) = \Sha(\Q, E')[\ell]\,.
\]
These cases require a second descent over $\phi' : E' \rightarrow E$, a full
$\ell$-descent on~$E$, or at the very least one must show that the elements
of order~$\ell$ in $\Sha(\Q, E')$ are not divisible by~$\ell$.
In a forthcoming paper of the first author with Brendan Creutz~\cite{CrMi},
these remaining cases are dealt with using a second descent over~$\phi'$.




\begin{thebibliography}{19}

\frenchspacing
\renewcommand{\baselinestretch}{1}

\bibitem{Beaver}
  C.~Beaver:
  {\em $5$-torsion in the Shafarevich-Tate group of a family of elliptic curves,}
  J. Number Th. {\bf 82} (2000), 25--46.

\bibitem{modularity2}
  C.~Breuil, B.~Conrad, F.~Diamond, R.~Taylor: {\em On the modularity of
  elliptic curves over $\Q$: wild $3$-adic exercises,} J. Amer. Math. Soc.
  {\bf 14} (2001), no. 4, 843--939.

\bibitem{VandVer}
  J.~Buhler, R.~Crandall, R.~Ernvall, T.~Mets{\"a}nkyl{\"a}:
  {\em Irregular primes and cyclotomic invariants to four million,}
  Math. Comp. {\bf 61} (1993), 151--153.

\bibitem{Cassels59}
  J.W.S.~Cassels: {\em Arithmetic on curves of genus 1. I. On a conjecture of Selmer,}
  J. reine angew. Math. {\bf 202} (1959), 52--99.

\bibitem{Cassels}
  J.W.S.~Cassels: \emph{Arithmetic on curves of genus 1. VIII. On conjectures
  of Birch and Swinnerton-Dyer}, J. reine angew. Math. {\bf 217} (1965), 180--189.

\bibitem{Crem}
  J.E.~Cremona: {\em Algorithms for modular elliptic curves,}
  Second edition, Cambridge University Press, Cambridge, 2007.

\bibitem{CremonaDB}
  J.E.~Cremona: Elliptic curves database, available online at \\
  \url{http://www.warwick.ac.uk/staff/J.E.Cremona/ftp/data/INDEX.html}

\bibitem{CremonaOnline}
  J.E.~Cremona: Online notes, available as item~26 at \\
  \url{http://www.warwick.ac.uk/staff/J.E.Cremona/papers}

\bibitem{Cremetal}
  J.E.~Cremona, T.A.~Fisher, C.~O'Neill, D.~Simon, M.~Stoll:
  {\em Explicit $n$-descent on elliptic curves, I. Algebra,}
  J. reine angew. Math. {\bf 615} (2008), 121--155;
  {\em II. Geometry,} J. reine angew. Math. {\bf 632} (2009), 63--84;
  {\em III. Algorithms,} preprint (2011), arXiv:1107.3516v1 [math.NT].

\bibitem{CrMi}
  B.~Creutz, R.L.~Miller:
  {\em Second isogeny descents and the Birch and Swinnerton-Dyer conjectural formula},
  preprint (2011), arXiv:1105.4018v1 [math.NT].

\bibitem{DeLong}
  M.~DeLong: {\em A formula for the Selmer group of a rational three-isogeny,}
  Acta Arith. {\bf 105} (2002), 119--131.

\bibitem{DokDokBSD}
  T.~Dokchitser, V.~Dokchitser: {\em On the Birch-Swinnerton-Dyer quotients
  modulo squares,} Ann. Math. {\bf 172} (2010), 567--596.

\bibitem{DokDokParity}
  T.~Dokchitser, V.~Dokchitser: {\em Parity of ranks for elliptic curves with a
  cyclic isogeny,} J. Number Th. {\bf 128} (2008), 662--679.

\bibitem{ElkRog}
  N.~Elkies, N.~Rogers: {\em Elliptic curves $x^3 + y^3 = k$ of high rank,}
  in: Algorithmic number theory,
  Springer Lect. Notes in Comp. Sci. {\bf 3076}, pp.~184--193, 2004.

\bibitem{TFisher}
  T.A.~Fisher: {\em On $5$ and $7$ descents for elliptic curves,}
  Ph.D. thesis, University of Cambridge (2000).

\bibitem{TFisherPap}
  T.A.~Fisher: {\em Some examples of $5$ and $7$ descent for elliptic curves over $\Q$,}
  J. Eur. Math. Soc. (JEMS) {\bf 3} (2001), 169--201.


\bibitem{FlynnGrattoni}
  E.V.~Flynn, C.~Grattoni:
  {\em Descent via isogeny on elliptic curves with large rational torsion subgroups,}
  J. Symb. Comp. {\bf 43} (2008), 293--303.

\bibitem{FGprog}
  E.V.~Flynn, C.~Grattoni: PARI programs, available at \\
  \url{http://people.maths.ox.ac.uk/flynn/genus2/flynngrattoni/}.

\bibitem{Frey84}
  G.~Frey: {\em Der Rang der L\"osungen von $Y^2 = X^3 \pm p^3$ \"uber $\Q$,}
  Manuscr. Math. {\bf 48} (1984), 71--101.


\bibitem{Kolyvagin}
  V.A.~Kolyvagin:
  \emph{Finiteness of $E(\Q)$ and $\Sha(E,\Q)$ for a subclass of Weil curves}
  (Russian), Izv. Akad. Nauk SSSR Ser. Mat. {\bf 52} (1988), no.~3, 522--540, 670--671;
  translation in Math. USSR-Izv. {\bf 32} (1989), no.~3, 523--541.

\bibitem{Mazur}
  B.~Mazur: \emph{Rational isogenies of prime degree}
  (with an appendix by D.~Goldfeld), Invent. Math. {\bf 44} (1978), no.~2, 129--162,.

\bibitem{RMiller}
  R.L.~Miller: {\em Proving the Birch and Swinnerton-Dyer conjecture for specific
  elliptic curves of analytic rank zero and one,}
  to appear in the LMS J. Comput. Math., arXiv:1010.2431v2 [math.NT].

\bibitem{Milne}
  J.S.~Milne: {\em Arithmetic duality theorems,}
  Perspectives in Mathematics, vol.~1, Academic Press, Orlando, Florida, 1986.

\bibitem{Mordell}
  L.J.~Mordell: {\em On the rational solutions of the indeterminate equations
  of the 3rd and 4th degrees}, Proc. Camb. Phil. Soc. {\bf 21} (1922), 179--192.

\bibitem{Nekovar}
  J.~Nekov\'a\v{r}:
  {\em Class number of quadratic fields and Shimura's correspondence,}
  Math. Ann. {\bf 287} (1990), 577--594.

\bibitem{Neukirch}
  J.~Neukirch: {\em Algebraische Zahlentheorie,}
  Springer-Verlag, Berlin Heidelberg, 1994.

\bibitem{JQuer}
  J.~Quer: {\em Corps quadratiques de $3$-rang $6$ et courbes elliptiques de rang $12$,}
  C. R. Acad. Sci. Paris (I) {\bf 305} (1987), 215--218.

\bibitem{PSatge}
  P.~Satg\'e: {\em Groupes de Selmer et corps cubiques,}
  J. Number Th. {\bf 23} (1986), 294--317.

\bibitem{Schaefer}
  E.F.~Schaefer: {\em Class groups and Selmer groups,}
  J. Number Th. {\bf 56} (1996), no.~1, 79--114.

\bibitem{SchaeferDescmap}
  E.F.~Schaefer:
  {\em Computing a Selmer group of a Jacobian using functions on the curve,}
  Math. Ann. {\bf 310} (1998), 447--471.

\bibitem{SchSt}
  E.F.~Schaefer, M.~Stoll: {\em How to do a $p$-descent on an elliptic curve,}
  Trans. Amer. Math. Soc. {\bf 356} (2004), 1209--1231.

\bibitem{Selmer51}
  E.S.~Selmer: {\em The diophantine equation $ax^3 + by^3 +cz^3 = 0$,}
  Acta Math. (Stockh.) {\bf 85} (1951), 203--362.

\bibitem{Selmer54}
  E.S.~Selmer: {\em The diophantine equation $ax^3 + by^3 +cz^3 = 0$.
  Completion of the tables,}
  Acta Math. (Stockh.) {\bf 92} (1954), 191--197.

\bibitem{Silverman}
  J.H.~Silverman: {\em The arithmetic of elliptic curves,}
  Springer Graduate Texts in Mathematics {\bf 106}, Springer-Verlag, New York Berlin
  Heidelberg Tokyo, 1986.

\bibitem{SilvermanII}
  J.H.~Silverman: {\em Advanced topics in the arithmetic of elliptic curves,}
  Springer Graduate Texts in Mathematics {\bf 151}, Springer-Verlag, New York
  Berlin Heidelberg, 1994.

\bibitem{JTop}
  J.~Top: {\em Descent by $3$-isogeny and $3$-rank of quadratic fields,}
  Advances in number theory (Kingston, ON, 1991), 303--317,
  Oxford Sci. Publ., Oxford Univ. Press, New York, 1993.

\bibitem{Velu}
  J.~V\'elu: {\em Isog\'enies entre courbes elliptiques,}
  C. R. Acad. Sci. Paris, S\'erie A {\bf 273} (1971), 238--241.

\bibitem{modularity1}
  A.J.~Wiles: {\em Modular elliptic curves and Fermat's last theorem,}
  Ann. of Math. (2) {\bf 2} (1995), no. 3, 443--551.

\bibitem{JWoo}
  J.~Woo:
  {\em Arithmetic of elliptic curves and surfaces: descents and quadratic sections,}
  Ph.D. thesis, Harvard University (2010).

\end{thebibliography}
\end{document}